\documentclass{article}
\usepackage{amsmath}
\usepackage{amssymb}

\newtheorem{theorem}{Theorem}

\newtheorem{conclusion}[theorem]{Conclusion}
\newtheorem{condition}[theorem]{Condition}

\newtheorem{corollary}[theorem]{Corollary}

\newtheorem{definition}[theorem]{Definition}

\newtheorem{lemma}[theorem]{Lemma}
\newtheorem{notation}[theorem]{Notation}

\newtheorem{proposition}[theorem]{Proposition}
\newtheorem{remark}[theorem]{Remark}

\newenvironment{proof}[1][Proof]{\noindent\textbf{#1.} }{\ \rule{0.5em}{0.5em}}
\input{tcilatex}
\begin{document}

\title{A new type of diagrams for modules}
\author{by Stephanos Gekas \\
Aristotle University of Thessaloniki, School of Mathematics}
\date{20 March 2016}
\maketitle

\begin{abstract}
We introduce a new type of diagrams and prove the existence of a particular
one, the "central tuned diagram", with some optimal features, for finitely
generated modules of certain categories. This is achieved by getting to the
idea of "the virtual category" of a module. Important applications are
specifically suggested to the modular representations of finite groups of
Lie type.
\end{abstract}

\section{\textbf{Introduction} \ }

\bigskip

We prove the existence of a "virtual diagram" for every module in certain
categories, for example that of $A$-modules of finite $k$-dimension over a
Frobenius $k$-algebra $A$.

This is done by proving the existence of a particular, well defined diagram,
to be called "the centrally tuned diagram", which is made so as to
illustrate the longest (in radical length) indecomposable summands ("longest
pillars") of radical sections (i.e., of sections of the type $Rad^{i}\left(
M\right) /Rad^{j}\left( M\right) $) as the connected components of the
relevant radical section of the diagram. The problem would not be so hard,
if there were no isomorphic idecomposable summands in radical sections, in
particular, isomorphic irreducibles in a radical layer: This is also why all
previous efforts to obtain some module diagrams, that we know of, have
essentially restricted themselves to this convenient special case.

To that purpose we introduce\ the notion of virtual radical (or, dually,
socle) series of $M$ (and, of course, of $A$ itself), in which the simple
constituents of the series are not just determined up to isomorphism but
elementwise (set-theoretically). The easiest way to conceive the idea is
through filtrations. Whenever we have a direct sum of isomorphic
indecomposables, a realization is an (elementwise) choice of an expression
of that sum; in the process of realization priority is given to the longest
pillars. The realization of the minor pillars (in radical subsections of the
radical sections of the major pillars) shall then be subjugated to the
realizations of the major ones; this is a key idea, which we refer to as
"vertical priority".\ \ 

After those coordinated ("tuned") "realizations", we appeal to the relevant
extension theory of $S$ by $N$ and its interpretation in terms of
homomorphisms in $Hom_{A}\left( \Omega ^{1}S,N\right) $ (for $A$ an artinian
ring or a Frobenius algebra), monitored by means of the second virtual
radical layer. In the same context we introduce the intuitively very
important notion of proportion class of an extension. That theory is briefly
reviewed with a fresh eye and some (may be) new insights in the next section.

All this can be much more neatly done in the frame of a new category we
define in section 3, "the virtual category of a module". That sheds a
completely new light and meaning in the radical (/socle) series, as well
enables us to get to "virtual diagrams" in the general case.

\bigskip Our thoughts with our module diagram extend to two directions: To
highlight, (a) the section structure / the possible filtrations, (b)
homomorphisms between modules, in particular automorphisms and
endomorphisms. The case of automorphisms is of a crucial importance, as any
homomorphism between modules may be viewed as the composition $\alpha \circ
h\circ \beta $, where $\alpha $, $\beta $ automorphisms ($\alpha $ of a
submodule-image, $\beta $ of a factor module-coimage) and $h$ a canonical
homomorphism, the elucidation of which is the same as that of submodules and
factor modules.

\bigskip By doing that, it shall be possible to get to complicated diagrams
from simpler subdiagrams, corresponding to their sections. In this respect,
the assertion on the existence of such a diagram and its properties shall be
of a crucial relevance.

\textit{We intend to investigate such possibilities and exemplify them
properly in future work. Among others, we wish to investigate, the
possibility of getting from relatively "nice" and simple diagrams for
injective modules of a reductive, simply connected semisimple algebraic
group }$G$\textit{\ over an algebraically closed field of positive
characteristic (in the defining characteristic) their restriction down to
the subgroup }$G\left( n\right) $\textit{\ of }$G$\textit{'s points over the
field with }$p^{n}$\textit{\ elements. This is discussed in the last
section. \ \ }

\begin{notation}
\bigskip For a ring $R$ we shall denote the $i$'th power Jacobson radical by 
$J^{i}\left( R\right) $, while for the i'th power Jacobson radical of a
(left, when nothing else is said) $R$-module $M$ we shall use both $%
J^{i}\left( R\right) M$ and $rad^{i}\left( M\right) $.
\end{notation}

\bigskip

\bigskip

\section{Ext, Hom traced on "virtual" radical (/socle) series}

\bigskip In this section, we approach the notion of "virtuality" in \textbf{%
radical }(resp. \textbf{socle}) \textbf{series} of a module
set-theoretically, where everything shall be understood as a submodule or
subsection of the given module, together with some implicit identification.

All this is more properly organized in the next section, by introducing the
concept of the "virtual category" , as the idea that naturally enables us
get a clear overview of the general module structure. In particular, we are
thus getting overview of the modules $Ext$, as we are already going to see
in this section for $Ext^{1}$, in a 1-1 correspondence to specific
set-theoretically specified (i.e., not just up to isomorphism) simple
modules on the second radical layer, in the case of p.i.m.'s (or dually the
second socle layer). That could previously somehow be done at most in case
that there were no isomorphic simple modules on that layer.

After taking the consequence of this specified insight by getting to the
idea of the virtual category, the next step is bound to be the idea of the
virtual diagram, which is to be the subject of the next section.

\begin{definition}
A \textbf{virtual radical }(resp. \textbf{socle}) \textbf{series} of a
module is one where the simple constituents on each layer are not just meant
as "isomorphic copies", but they have also been specified set-theoretically.
\end{definition}

\bigskip

Notice that the subsequent orders ("powers") of the radical, respectively of
the socle, are submodules, thus set-theoretically well defined (with respect
to the original module). However, when taking the quotient of radicals (%
\textit{resp.} socles) of subsequent orders, in order to form the layers of
the series, one defines them just up to isomorphism. We wish here to break
with this tradition: \textbf{To this end, we consider the constituents of
each layer set-theoretically as consisting of cosets of the next order
radical (resp., of the previous order socle).} This subtle differentiation
opens up for new insights and possibilities, as I am beginning to exhibit in
this paper.

This set-theoretical specification at all layers necessitates some \textbf{%
identifications of sets,} which in this section are just introduced
intuitively, but which shall finally lead us safely to the definition of the
virtual category in the next section.

\bigskip 

\begin{condition}
\bigskip By viewing the elements (of the irreducibles of) each radical layer
as the biggest possible (: down to the last layer) cosets, it becomes
possible to project them to the lesser cosets (by factoring out a
submodule), enabling us to IDENTIFY the subsections of (f.ex. radical)
radical sections to their canonically isomorphic sections, in particular
identify any isomorphic pairs of the type

$rad^{i}M/rad^{j}\left( M\right) $, $\left( rad^{i}M/rad^{j+\lambda
}M\right) \diagup \left( rad^{j}M/rad^{j+\lambda }M\right) $ $\left(
i<j\right) $ as well as their contents in a well-defined $1-1$
correspondence. 

We are taking this identification for given throughout this work; it is,
however, more proper to define a suitable category, where also these
identifications are fully put into their right context. That is being done
in the next section.
\end{condition}

\bigskip 

We note further that the "virtuality" of the contents of the radical layers
is also conceivable through suitable (: compatible, i.e. \textbf{realizable}%
, see below) filtrations, a point that is important to keep in mind. \textbf{%
Indeed, virtual (radical/socle) series (and the corresponding virtual
diagrams, which we are going to define) are, at least in case not all simple
copies on the same layer are non-isomorphic, the only kind of picture of a
module, that may properly enable us to follow and overview filtrations.}

\bigskip\ It proves also expedient to our goals to introduce a new notion of
a more confined inverse image, in cases of canonical epimorphisms from
direct sums.

In our new, "virtual category" we are going to identify two (plus one) kinds
of pairs of naturally isomorphic subsections of a given module $M$: 

a. Pairs "of type q" $\left\{ A/B,\ A/C\diagup B/C\right\} $ and

b. Pairs "of type s" $\left\{ \left( \tbigoplus\limits_{i=1}^{n}M_{i}\right)
\diagup \left( \tbigoplus\limits_{i=1}^{n}S_{i}\right) ,\
\tbigoplus\limits_{i=1}^{n}M_{i}/S_{i}\right\} $, to which we shall add all
pairs "of type S"\ $\left\{ \widetilde{\pi }^{-1}\left( N\right) ,\varpi
^{-1}\left( N\right) \right\} $, i.e. preimages of canonical epimorphisms
from direct sums and what we shall be calling "confined preimages",\ which
will mean following :

Let $M=\tbigoplus\limits_{i=1}^{n}M_{i}$ be a certain decomposition of the
section $M$ of a given module $K$\ as a direct sum of indecomposables with
each $M_{i}$ posessing a certain submodule $S_{i}$; consider the natural
isomorphism $\sigma :\left( \tbigoplus\limits_{i=1}^{n}M_{i}\right) \diagup
\left( \tbigoplus\limits_{i=1}^{n}S_{i}\right) \longrightarrow \
\tbigoplus\limits_{i=1}^{n}M_{i}/S_{i}$,\ and the canonical epimorphisms $%
\varpi :M\longrightarrow M\diagup \left(
\tbigoplus\limits_{i=1}^{n}S_{i}\right) $ and $\pi _{i}:M_{i}\longrightarrow
M_{i}/S_{i}$. Further, for any subset $J$=$\left\{ j_{1},...,j_{s}\right\} $
of $\left\{ 1,...,n\right\} $ define $\pi _{J}:=\left( \pi _{j_{1}},...,\pi
_{j_{s}}\right) :\tbigoplus\limits_{j\in J}^{{}}M_{j}\longrightarrow
\tbigoplus\limits_{j\in J}^{{}}M_{j}/S_{j}$ and let $p_{J}$\ be the usual
direct sum projection corresponding to the index subset $J$, $%
\tbigoplus\limits_{i=1}^{n}M_{i}/S_{i}\longrightarrow
\tbigoplus\limits_{j\in J}^{{}}M_{j}/S_{j}$.\ 

Assume, now, $N$ to be any submodule of $M\diagup \left(
\tbigoplus\limits_{i=1}^{n}S_{i}\right) $. Notice that, if we are \textit{%
"inside of a module }$K$\textit{"}, i.e. if $M$ is a section of a module $K$%
, then $M\diagup \left( \tbigoplus\limits_{i=1}^{n}S_{i}\right) $ and $%
\tbigoplus\limits_{i=1}^{n}M_{i}/S_{i}$ are being "virtually" (i.e., in the
virtual category $\hat{W}_{K}$ of $K$) identified (as pairs "of type s"). $\
\ \ \ \ \ \ \ \ \ \ \ \ \ \ \ \ \ \ \ $

\begin{definition}
\bigskip Let $J$\ be the subset of $\left\{ 1,...,n\right\} $\ which is
minimal with the property that $\sigma \left( N\right) $ be contained in $%
\tbigoplus\limits_{j\in J}^{{}}M_{j}/S_{j}$ (equivalently, that $%
res_{N}\left( \pi _{J}\circ \sigma \right) $\ is injective). We define "%
\textbf{the confined }$\varpi $\textbf{-preimage} $\widetilde{\varpi }%
^{-1}\left( N\right) $ of $N$" as, $\widetilde{\varpi }^{-1}\left( N\right)
:=\left( \pi _{J}^{-1}\circ p_{J}\circ \sigma \right) \left( N\right) $\ \ \
\ (being a submodule of $\tbigoplus\limits_{j\in J}^{{}}M_{j}/S_{j}$). 
\end{definition}

\bigskip In the following we shall just be using the notation $\widetilde{%
\varpi }^{-1}\left( N\right) $\ for it, without further notification
whenever the context is clear. Take care of the important fact, that this
preimage depends on the choice of a decomposition $M=\tbigoplus%
\limits_{i=1}^{n}M_{i}$ into indecomposables (but not really on the $\varpi -
$kernel $\tbigoplus\limits_{i=1}^{n}S_{i}$), or at least of the direct
summand $\tbigoplus\limits_{j\in J}^{{}}M_{j}$, unless all the
indecomposables $M_{i}$, $i=1,...,n$, are non-isomorphic, in which case the
definition remains unambiguous, with no need for further specification.\ \ 

\bigskip The direct sums that are relevant in our context here are just ones
that appear as sections of the module.

\bigskip 

\bigskip In order to get full advantage of virtual series, we must find an
optimal and coordinated way of choosing the direct summands on each layer in
cases of more than one isomorphic copies of irreducibles. This shall
eventually lead us to the natural idea of a virtual diagram depicting the
structure of the module. 

We begin here with some preparatory revision and reinforcement (in virtual
terms) of some very closely related facts.

In what follows we shall be considering a Frobenius $k$-algebra $A$ (for
example a finite-group algebra) and by "modules" we mean left $A$-modules
(always finite dimensional as $k$-vector spaces). Of course, we are thinking
of group algebras over fields of characteristic $p\neq 0$ dividing the order
of the group, since otherwise, according to Maschke's theorem, the group
algebra be semisimple - and hence uninteresting to us.

Let $V$ be an $A$-module and consider a virtual radical (and a virtual
socle) series of it.

\bigskip

\begin{proposition}
If $\phi $\ is a non-zero ($A$-)homomorphism $VJ\left( A\right)
\longrightarrow N$, \ then $\phi $ cannot factor through a projective
A-module.
\end{proposition}

\begin{proof}
\textit{First proof:}

By isolating the relevant summand, we may reduce to the case that $V$ is
indecomposable.

Let, so, $\phi \ \func{factor}\ $through $P$: $\phi =ba$, where $%
b:P\longrightarrow N$\textit{.}

We are first proving that then, $\phi $\ does also factor through \ $V$'s
injective hull, say $Q$. Consider $V$ embedded as a submodule of $Q$;\ then,
apparently, $J\left( A\right) V\subseteq J\left( A\right) Q$.

$%
\begin{array}{ccc}
Q & \psi -\rightarrow & P \\ 
\iota \uparrow & a\nearrow & \downarrow b \\ 
VJ\left( A\right) & _{\phi }\twoheadrightarrow & N \\ 
\uparrow &  &  \\ 
0 &  & 
\end{array}%
$ \ \ \ Due to $P$ 's injectivity we get $\psi :Q\longrightarrow P$, such
that $\psi \iota =a$, hence $\left( b\psi \right) \iota =$\ $\phi $.

Since $N$ is simple, $b\psi :Q\longrightarrow N$ suggests $QJ\left( A\right)
\subseteq \ker b\psi $, hence also $VJ\left( A\right) \subseteq \ker b\psi $%
, which necessarily means $\phi =0$, a contradiction.

\textit{Second proof: }

$\phi $ induces an isomorphism $N_{1}\longrightarrow N$ for some simple
submodule $N_{1}$ of $J\left( A\right) V/J^{2}\left( A\right) V$. In case $%
\phi $\ factors through a projective $P$, say by $\iota :J\left( A\right)
V\longrightarrow P$, that induced isomorphism $N_{1}\widetilde{%
\longrightarrow }N$ necessarily factors through the one induced from $\func{%
Im}\left( \iota \right) $'s head. But since b induces an epimorphism from $P$%
's head onto $N$, it becomes clear that $\func{Im}\left( \iota \right) $'s
head lies in $P$'s head. Hence $P$'s head contains an isomorphic copy of $N$.

We are first proving that $\phi $\ does also factor through \ $N$'s
projective cover, say $P_{N}$.

$%
\begin{array}{ccc}
P & --\rightarrow & P_{N} \\ 
\iota \uparrow & b\searrow & \downarrow \\ 
VJ\left( A\right) & _{\phi }\twoheadrightarrow & N \\ 
&  & \downarrow \\ 
&  & 0%
\end{array}%
$ \ In case $P$ is not isomorphic to the projective cover of $N$, then its
head, as noticed containing an isomorphic copy of $N$, must also contain
more irreducibles than that (see f.ex. \cite[6.25(i), p. 133]{CR}). On the
other hand, by $P$'s projectivity we get a homomorphism $P\longrightarrow
P_{N}$, which is surjective, since it induces an epimorphism onto $P_{N}$'s
simple head $N$, and which must then split, due to $P_{N}$'s projectivity: $%
P=P_{0}\oplus P_{N}^{^{\prime }}$, with $P_{N}^{^{\prime }}\cong P_{N}$ and $%
P_{0}=\ker \left( \ P\longrightarrow P_{N}\right) $, another projective, as
a direct summand of a projective. For the same reason the composite $\sigma
: $ $J\left( A\right) V\longrightarrow P_{N}$ is an epimorphism (as inducing
an isomorphism $N_{1}\widetilde{\longrightarrow }N$ onto the head of $P_{N}$%
), thus split, yielding\ $VJ\left( A\right) =\ker \sigma \oplus M$, with $%
M\cong P_{N}$, thus giving a composite monomorphism $\zeta
:P_{N}\rightarrowtail J\left( A\right) V\rightarrowtail V$, which has to
split, as $P_{N}$ is injective too, since $A$ has been assumed to be
(symmetric, hence also) (quasi-)Frobenius. But this does also imply that the 
$Hd\left( \zeta \left( P_{N}\right) \right) =N_{1}$ (notice also here the "$%
= $" instead of "$\cong $"!) is a direct summand in $Hd\left( V\right) $,
contrary to the fact that $N_{1}$ is actually in the second layer of the
virtual radical series of $V$.
\end{proof}

\begin{remark}
\bigskip We wish here to emphasize that, by saying " $\phi :M\longrightarrow
N$ factors through a projective" is meant that it may be written as a
composition with some $\beta :P\longrightarrow N$, i.e. from $P$ to $N$!
\end{remark}

\bigskip

\bigskip We remind the reader that the Heller operators $\Omega ^{n}V=\ker
\left( \partial _{n-1}:P_{n-1}\longrightarrow P_{n-2}\right) $ (for $n>0$
and a minimal projective resolution of $V$ - dually for $n<0$) may only be
defined in a category of finitely generated $A$-modules, where the
Krull-Schmidt theorem holds and hence a unique minimal projective resolution
is available; to this end, it is actually sufficient to assume that the ring 
$A$ is (left-) artinian (see \cite{CR} for a proof), an assumption that we
are keeping to here. Otherwise these operators are in general only definable
in the stable category, i.e. up to a projective direct summand.

\begin{lemma}
\bigskip If $\phi $\ is a non-zero ($A$-)homomorphism $\Omega
^{n}V\longrightarrow N$ ($A$ a Frobenius $k$-algebra, $N$ simple, $n>0$), \
then $\phi $ cannot factor through a projective $A$-module.
\end{lemma}

\begin{proof}
\bigskip By taking a minimal projective resolution of V, $...\longrightarrow
P_{1}\longrightarrow P_{0}\longrightarrow V\longrightarrow 0$ , $\Omega
^{n}V $ is isomorphic to $\ker \left( \partial _{n-1}:P_{n-1}\longrightarrow
P_{n-2}\right) $, embedded in $P_{n-1}$ through, say, $\iota $. Appealing to
the Frobenius property, $P$\ is injective too, which forces $a$ to factor
through $\psi :P_{n-1}\longrightarrow P$. We may then again proceed in two
ways:

\textit{First proof:}\ 

$%
\begin{array}{ccc}
P_{n-1} & \psi -\rightarrow & P \\ 
\iota \uparrow & a\nearrow & \downarrow b \\ 
\Omega ^{n}V & _{\phi }\twoheadrightarrow & N \\ 
\uparrow &  &  \\ 
0 &  & 
\end{array}%
$ \ Due to the minimality of the resolution, $\Omega ^{n}V$ $\subseteq
J\left( A\right) P_{n-1}$; on the other hand, since $N=\func{Im}\left( b\psi
\right) $ is simple, we get $J\left( A\right) P_{n-1}\subseteq $ $\ker
\left( b\psi \right) $, therefore also is $b\psi \iota $ (meaning here, the
restriction of $b\psi $ to $\Omega ^{n}V$) equal to 0, hence $\phi =0$,
contrary to our assumption.

\ \textit{Second proof: }

By arguing in the same way as in the second proof of the proposition above,
we get by assuming that $\phi $\ factors through a projective $P$, that $%
\Omega ^{n}V$ has a projective direct summand isomorphic to the projective
cover of $N$, which contradicts the minimality of our resolution.
\end{proof}

\bigskip

By dualizing (either by starting with a minimal injective resolution and
mirror the above proof or by using the isomorphism $\Omega ^{-n}V\cong
\Omega ^{n}\left( V^{\ast }\right) ^{\ast }$ and the lemma), we get the
following

\begin{corollary}
\bigskip If $\phi $\ is a non-zero ($A$-)homomorphism $N\longrightarrow
\Omega ^{-n}V$ ($A$ left artinian, $N$ simple, n\TEXTsymbol{>}0), then $\phi 
$ cannot factor through a projective $A$-module.
\end{corollary}

\bigskip Let now $A$ be a Frobenius algebra, in which case the category $_{A}%
\func{mod}$ of (left) $A$-modules is a Frobenius category.

Furthermore we are now for a moment shifting our contemplation from the
category $_{A}\func{mod}$ into the (triangulated) stable category $_{A}st%
\func{mod}$, which still has the same objects as $_{A}\func{mod}$ (although
we are now getting much bigger/fewer isomorphism classes of them), but whose
morphisms \underline{$Hom$}$_{A}\left( ,\right) $ are obtained as
equivalence classes in $Hom_{A}\left( ,\right) $ modulo the "ideal" (in the
sense of an additive category) $PHom(,)$, consisting of those homomorphisms,
that factor through a projective module; i.e., given two $A$-modules $M$, $N$%
, \underline{$Hom$}$_{A}\left( M,N\right) =Hom_{A}(M,N)/PHom_{A}(M,N)$. This
is a triangulated category with translation functor $T=\Omega ^{-1}$. A very
important point in considering the stable category is the fact that 
\begin{equation}
Ext_{A}^{n}\left( M,N\right) \cong \ \underline{Hom}_{A}\left( \Omega
^{n}M,N\right) \cong \ \underline{Hom}_{A}\left( M,\Omega ^{-n}N\right)
\end{equation}

for any $n\in 
\mathbb{Z}
$ (See f.ex. \cite[5.1, also 5.2 \& 4.4(v)]{JC}).

Up to this point it would be enough to have a Frobenius category, such as
the category $_{A}\func{mod}$ of modules over a Frobenius algebra $A$. If we
now specialize in the case, where $A$ is a group algebra $kG$ over a finite
group $G$, which also is our main motivation, then we have an algebra that
is not only artinian and Frobenius, but also symmetric. For a symmetric $k$%
-algebra $A$ the p.i.m.'s (principal indecomposable modules, i.e. the
projective covers [being here the same as the injective hulls] of its simple
modules) have isomorphic head and socle (see f.ex. \cite[I 7.5(iii)]{PL} );
for a more general category-theoretic set-up and approach to the stable
category and related topics see f.ex. \cite{HK}, for more specific details
as to the triangulization of the stable category of modules over a
finite-group algebra see \cite{JC}.

\bigskip \bigskip Our Lemma with its Corollary above, together with relation
(1), imply the following

\begin{proposition}
Given a Frobenius algebra $A$, the $A$-modules $M$, $N$, where $N$ is
simple, for $n>0$ we have $Ext_{A}^{n}\left( M,N\right) \cong Hom_{A}\left(
\Omega ^{n}M,N\right) $ and $Ext_{A}^{n}\left( N,M\right) \cong
Hom_{A}\left( N,\Omega ^{-n}M\right) $.
\end{proposition}

\begin{proof}
Our Lemma, resp. its Corollary, gives $PHom_{A}(\Omega ^{n}M,N)=0$ and,
respectively, $PHom_{A}\left( N,\Omega ^{-n}M\right) =0$\ - which, by virtue
of the two relations (1), yield the results.
\end{proof}

\begin{remark}
\bigskip We note that, by following upon the line of the proof of the above
Lemma 5, we can show the proposition indepedently of the relations (1), for
any artinian algebra; see for example \cite[2.5.4]{DB}.
\end{remark}

We are going to pursue this correspondence further, by following it along
the virtual series in detail: \textit{a possibility which is only there, if
we are considering the virtual instead of the usual series.}

\bigskip

\bigskip

\bigskip

\begin{definition}
\bigskip We call two extensions $0\longrightarrow N_{\kappa }\longrightarrow
B_{\kappa }\longrightarrow S\longrightarrow 0$, $\kappa =1,2$,
"proportional" if there are isomorphisms $\sigma :N_{1}\longrightarrow N_{2}$
and $B_{1}\longrightarrow B_{2}$, such that the following diagram commute: \ 
$%
\begin{array}{ccccccccc}
0 & \rightarrow & N_{1} & \longrightarrow & B_{1} & \longrightarrow & S & 
\rightarrow & 0 \\ 
&  & \downarrow \sigma &  & \downarrow &  & \parallel &  &  \\ 
0 & \rightarrow & N_{2} & \longrightarrow & B_{2} & \longrightarrow & S & 
\rightarrow & 0%
\end{array}%
$ We shall call $\sigma $\ their "proportion"; in case $N$ is simple, $%
\sigma $\ corresponds to a non-zero element of the field $k$, as then $%
Aut_{A}\left( N\right) \cong k^{\ast }$ (Schur's Lemma). \ \ \ \ \ \ \ \ \ \
\ \ \ \ \ \ \ \ \ \ \ \ \ \ \ \ \ \ \ \ \ \ \ \ \ \ \ \ \ \ \ \ \ \ \ \ \ \
\ \ \ \ \ \ \ \ \ \ \ \ \ \ \ \ \ \ \ \ \ \ \ \ \ \ \ \ \ \ \ \ \ \ \ \ \ \
\ \ \ 
\end{definition}

\bigskip Notice that in this definition we did not assume $N_{\kappa }\
\left( \kappa =1,2\right) \ $and $S$ to be simple; however we are primarily
going to apply the concept in that case. This defines an equivalence
relation on the set of extensions of $S$ by $N$, the classes of which we
shall call "proportionaliy classes".

\begin{notation}
\bigskip\ Let now $V$ be an $A$-module, with simple head S (forcing $V$ to
be indecomposable) and let its second layer of virtual radical series be $%
VJ\left( A\right) /VJ^{2}\left( A\right)
=\tbigoplus\limits_{i=1}^{r}N_{i}\oplus W$ ($J\left( A\right) $ denotes the
Jacobson radical of $A$), where $N_{1}\cong N_{2}\cong ...\cong N_{r}\cong N$
simple $A$-modules and $W$ is a sum of irreducibles, all non-isomorphic to $%
N $. Notice that we have put \textbf{"=" and not "}$\cong $\textbf{":} that
is due to our set-theoretically \textit{virtual} approach.\ For $\kappa
=1,...,r $ we get a factor module $B_{\kappa }=%
\begin{array}{c}
S \\ 
N_{\kappa }%
\end{array}%
$, which we may specify more precisely as follows: Let $VJ\left( A\right)
/VJ^{2}\left( A\right) =N_{\kappa }\oplus R_{\kappa }$ and let $\pi
:VJ\left( A\right) \twoheadrightarrow VJ\left( A\right) /VJ^{2}\left(
A\right) $ be the natural epimorphism. Then $B_{\kappa }:=V/\pi ^{-1}\left(
R_{\kappa }\right) $.
\end{notation}

\bigskip As it is well-known, the elements of $Ext_{A}^{n}\left( M,N\right) $
are in 1-1 correspondence with the set of equivalence classes of
n-extensions of $M$ by $N$ (see \cite[Th. 6.3, p. 29]{JC}). Due to the
isomorphism \ \ above, we do also get an 1-1 correspondence between $%
\underline{Hom}_{A}\left( \Omega ^{n}M,N\right) $ and the set of equivalence
classes of $n$-extensions of $M$ by $N$; that correspondence is in terms of $%
\underline{Hom}_{A}\left( \Omega ^{n}M,N\right) $ precisely specified in 
\cite[2.6, p. 40]{DB}. We shall mainly be concerned with this correspondence
for $n=1$ and for modules over a symmetric algebra, in which case we also
have Proposition 7.

Adjusting to the standard notation, call $\zeta _{\kappa }$ the element of $%
Ext_{A}^{1}\left( M,N\right) $,\ which corresponds to the canonical
surjection $\widehat{\zeta }_{\kappa }:\Omega ^{1}\left( S\right) $ $%
\twoheadrightarrow N_{\kappa }$, and let $L_{\kappa }$ designate its kernel.

Let us now consider the standard minimal projective resolution of $S$, where
we will denote the projective cover os $S$ as $P_{0}$, with the whole
resolution taken in a set-theoretically virtual way, thus enabling us f.ex.
to consider $P_{0}/P_{0}J\left( A\right) $ (now, of course, $\Omega ^{1}S$
is the same as $P_{0}$'s radical $P_{0}J\left( A\right) $) as \textbf{equal}
and \textbf{not just isomorphic} to $S$ - and so on, for the whole complex!
It may appear wierd doing so, while moving "up the hill" in subsequent
epimorphisms (which we want to be able to identify as "the natural ones",
i.e. as gotten by taking the quotients by submodules, which allows for
virtuality), but we may always start up at the necessary level and move
"downhill", down to modules consisting of (ever bigger) cosets.

Let us use the notation \ \ above for the module $P_{0}$\ \ in place of $V$,
where we now have $L_{\kappa }=\pi ^{-1}\left( R_{\kappa }\right) $, in the
above notation.

$\ 
\begin{array}{cccccccccccccccc}
\rightarrow & P_{1} & \rightarrow & P_{2} & \underrightarrow{\partial _{2}}
& \_P_{1} & \_\_\underline{} & \_\underline{\partial _{1}}\_ & 
\underrightarrow{} & P_{0} & \longrightarrow & S & \rightarrow & 0 &  &  \\ 
&  &  & \mid &  & \_\mathbf{\mid } & q\searrow &  & \nearrow & \mathbf{\mid }
&  &  &  &  &  &  \\ 
&  &  & \mid &  & \mu _{1}\mathbf{\mid } &  & \Omega ^{1}S &  & \mathbf{\mid 
} & \mu _{0}\_ & \parallel &  &  &  &  \\ 
&  &  & \downarrow &  & \_\mathbf{\downarrow } & \swarrow \widehat{\zeta }%
_{\kappa } &  &  & \mathbf{\downarrow } &  &  &  &  &  &  \\ 
&  &  & 0 & \longrightarrow & \_N_{\kappa } & -- & --- & \rightarrow & 
B_{\kappa } & \longrightarrow & S & \rightarrow & 0 &  & 
\end{array}%
$

We know that any two different homomorphisms $\widehat{\zeta }_{\kappa
}:\Omega ^{1}S\longrightarrow N$, i=1,2, represent different classes in $%
Ext_{A}^{1}\left( S,N\right) $, hence also non-equivalent extensions; of
course, those different homomorphisms are in 1-1 correpondence to the
homomorphisms $\overline{\zeta }_{\kappa }:Hd\left( \Omega ^{1}S\right)
\longrightarrow N$, which are in turn up to an automorphism of $N$ (meaning
here up to an automorphism of $N_{\kappa }$) determined by $\ker \left( 
\overline{\zeta }_{\kappa }\right) $. But "up to an automorphism of $%
N_{\kappa }$" means "up to a proportion", in view of our definition; i.e.
each direct summand in $Hd\left( \Omega ^{1}S\right) $\ determines a
proportionality class. Hence we get the following

\bigskip

\begin{proposition}
\bigskip With the above notation, two homomorphisms in $Hom_{A}\left( \Omega
^{1}S,N\right) $ correspond to proportional extensions of $S$ by $N$
(equivalently, they correspond to each other's multiple in $%
Ext_{A}^{1}\left( S,N\right) $) iff their induced homomorphisms $Hd\left(
\Omega ^{1}S\right) \longrightarrow N$ have the same kernel. We may thus
have a precisely determined $A$-homomorphic correspondence $%
Ext_{A}^{1}\left( S,N\right) \widetilde{\leftrightarrow }Hom_{A}\left(
Hd\left( \Omega ^{1}S\right) ,N\right) \cong k^{r}$ (by Schur's lemma),
where $r$ is the number of \ direct summands of $Hd\left( \Omega
^{1}S\right) $, that are isomorphic to $N$, whereby the propotionality class
of an extension $B$ of $P_{0}$'s head $S$ by $N$ is determined by the direct
summand isomorphic to $N$ in $Hd\left( \Omega ^{1}S\right) $ that is
involved (: is not in the kernel).
\end{proposition}

\bigskip\ 

\begin{proposition}
\bigskip For a module $V$ as above, i.e. with simple head $S$ and with $%
J\left( A\right) V/J^{2}\left( A\right)
V=\tbigoplus\limits_{i=1}^{r}N_{i}\oplus W$, where $N_{1}\cong N_{2}\cong
...\cong N_{r}\cong N$ simple $A$-modules and $W$ is a sum of irreducibles,
all non-isomorphic to $N$, the elements in $Ext_{A}^{1}\left( S,N\right) $
corresponding to the equivalence classes of the $r$ extensions $%
0\longrightarrow N_{\kappa }\longrightarrow B_{\kappa }\longrightarrow
S\longrightarrow 0$ $~~$($N_{\kappa }\cong N$) are again $k$-linearly
independent.
\end{proposition}

\begin{theorem}
Conversely, if a module $V$ with the simple $S$ in its head has $r$ factor
modules of the type $%
\begin{array}{c}
S \\ 
N_{\kappa }%
\end{array}%
$, with $N_{\kappa }\cong N$, $\kappa =1,...,r$, with that $S$ on the head,
so that these $r$ extensions correspond to $r$\ $\ k$-linearly independent
elements in $Ext_{A}^{1}\left( S,N\right) $, then we have an injection $%
\tbigoplus\limits_{i=1}^{r}N_{\kappa }\hookrightarrow J\left( A\right)
V/J^{2}\left( A\right) V$ ,i.e., we get to a non-split extension $%
0\longrightarrow \tbigoplus\limits_{i=1}^{r}N_{\kappa }\longrightarrow
B\longrightarrow S\longrightarrow 0$ , where $B$ is a factor module of $V$.
\end{theorem}

\begin{proof}
\bigskip \bigskip By pointing out that we have an epimorphism onto $V$ from
the projective cover $P_{S}$ of $S$, we may substitute $P_{S}$ for $V$ and
study the lifts of these extensions there; the reason why this epimorphism
exists is that the projective cover of $V$ is just $P_{S}$, i.e. the same as
the projective cover of its head $S$ (see \cite[6,23 or 6.25 (ii)]{CR} ).
Speaking in terms of virtual radical layers, we may lift the chosen $%
N_{\kappa }$'s along $P_{S}\twoheadrightarrow V$, set theoretically and
elementwise, to ones now consisting of bigger cosets as elements, \textbf{in
a procedure that clearly respects the extension classes} of the $B_{\kappa }$%
's.

The converse is clear.
\end{proof}

\bigskip This proposition may of course be dualized, to get a similar one in
terms of the virtual socle series:

\bigskip \bigskip

\bigskip

\section{Module Diagrams}

\bigskip 

As we intuitively did in the previous section, so also in our new, "virtual
category" of a given module $K$ we are going to identify two (plus one)
kinds of pairs of naturally isomorphic subsections of the given module: 

a. Pairs "of type q" $\left\{ A/B,\ A/C\diagup B/C\right\} $ and

b. Pairs "of type s" $\left\{ \left( \tbigoplus\limits_{i=1}^{n}M_{i}\right)
\diagup \left( \tbigoplus\limits_{i=1}^{n}S_{i}\right) ,\
\tbigoplus\limits_{i=1}^{n}M_{i}/S_{i}\right\} $, to which we shall add all
pairs "of type S"\ $\left\{ \widetilde{\pi }^{-1}\left( N\right) ,\varpi
^{-1}\left( N\right) \right\} $, i.e. preimages of canonical epimorphisms
from direct sums and their "confined preimages".

\bigskip The direct sums that are going to appear inside our virtual
category are only sums that appear as sections of the module; otherwise we
can apparently not define sums in general in such a category.

\begin{definition}
Given a module $K$, we define its "virtual category" \bigskip $\hat{W}_{K}$,
whose objects are obtained from the family of all of $K$' s sections,
sections of sections and so on, after identifying  (also elementwise meant)
all pairs "of type q" $\left\{ A/B,\ A/C\diagup B/C\right\} $ and also all
pairs "of type s" $\left\{ \left( \tbigoplus\limits_{i=1}^{n}M_{i}\right)
\diagup \left( \tbigoplus\limits_{i=1}^{n}S_{i}\right) \text{,}\
\tbigoplus\limits_{i=1}^{n}M_{i}/S_{i}\right\} $ of naturally isomorphic
subsections of $K$, as well as pairs "of type $S$"\ $\left\{ \widetilde{%
\varpi }^{-1}\left( N\right) ,\pi ^{-1}\left( N\right) \right\} $\ (see
previous definition),\ and in which morphisms are the ones induced by the
module homomorphisms between the sections.
\end{definition}

\bigskip It is immediate to see that $\hat{W}_{K}$\ is an exact category.

We shall usually denote the objects of the virtual category $\hat{W}_{K}$\
(to be called "virtual sections" of $K$) with the same letters as the
sections themselves, although we shall be considering them as objects of the
virtual category of $K$, indeed corresponding to equivalence classes of
subsections, resulting from the identifications that we applied on the
sections of $K$. This convention shall always be implied in the following,
with no further notification.

We may also notice that in the case of identifications of type $S$ it is
more handy to use the confined preimages as representatives of the virtual
sections in question.

In fact the previous chapter, where we have defined "virtuality" in a
set-theoretic way (which may easily be shown to be equivalent to that of the
virtual category), may also be viewed in the framework of the virtual
category.

Let now $A$ be a Frobenius $k$-algebra from now on.

As we have seen (Theorem 12), whenever we have an indecomposable with simple
head, that simple head forms non-split extensions with all virtual
irreducibles on the next radical layer, extensions which are also $k$%
-linearly independent in case they are isomorphic; further, if the
indecomposable is projective (i.e., a p.i.m.), those extensions correspond
to a $k$-basis of the corresponding module of type $Ext_{A}^{1}\left(
,\right) $.

We shall in the following be working in a suitable category of modules,
which may either be that of the finitely generated modules over such a
Frobenius $k$-algebra $A$, or of the finitely generated rational $G$%
-modules, where $G$ is a reductive algebraic group, or even a truncated
subcategory $C\left( \pi \right) $ of that, for some saturated subset $\pi $
of $X\left( T\right) _{+}$ \cite[chapter A]{JCJ}. Whenever we come to
representations of algebraic groups, if not specified, our notation shall be
that of \cite{JCJ}.

So henceforth, whenever we just refer to "a module", if not specified, we
shall mean a module in any of these categories. \textbf{When working with
(sub)sections of a particular module, that shall automatically be done in
the context of its virtual category, in which subsequently all equalities
have to be understood.}

It is very important for us, that also in the case of representations of
reductive algebraic groups we do have some propositions quite similar to the
last two propositions (of section 2), in which we may now use the Weyl
modules and their duals instead of the p.i.m.'s that were used for modules
over a Frobenius algebra (see f.ex. \cite[II, 2.12(4) combined with 2.14
prop.\&(4)]{JCJ}) - leading, by a similar to the above procedure, to:

\begin{proposition}
\bigskip For $\lambda $,$\mu \in X\left( T\right) _{+}$, with $\mu \ngtr
\lambda $, we have a precisely determined $G$-homomorphic correspondence $%
Ext_{G}^{1}\left( L\left( \lambda \right) ,L\left( \mu \right) \right) 
\widetilde{\leftrightarrow }Hom_{G}\left( radV\left( \lambda \right)
,L\left( \mu \right) \right) $

$\widetilde{\leftrightarrow }Hom_{G}\left( radV\left( \lambda \right)
\diagup rad^{2}V\left( \lambda \right) ,L\left( \mu \right) \right) 
\widetilde{\leftrightarrow }k^{r}$ for some $r\geq 0$, as well as

$Ext_{G}^{1}\left( L\left( \mu \right) ,L\left( \lambda \right) \right) 
\widetilde{\leftrightarrow }Hom_{G}\left( L\left( \mu \right) ,H^{0}\left(
\lambda \right) \diagup socH^{0}\left( \lambda \right) \right) \widetilde{%
\leftrightarrow }$

$\widetilde{\leftrightarrow }Hom_{G}\left( L\left( \mu \right)
,soc^{2}H^{0}\left( \lambda \right) \diagup socH^{0}\left( \lambda \right)
\right) \widetilde{\leftrightarrow }k^{r}$, while also $Ext_{G}^{1}\left(
L\left( \lambda \right) ,L\left( \mu \right) \right) \equiv
Ext_{G}^{1}\left( L\left( \mu \right) ,L\left( \lambda \right) \right) $.
\end{proposition}

\ 

\ \ \ \ \ \ \ \ \ \ \ \ \ \ \ \ \ \ \ \ \ \ \ \ \ \ \ \ \ \ \ \ \ \ \ \ \ \
\ \ \ \ \ \ \ \ \ \ \ \ \ \ \ \ \ \ \ \ \ \ \ \ \ \ \ \ \ \ \ \ \ \ \ \ \ \
\ \ \ \ \ \ \ \ \ \ \ \ \ \ \ \ \ \ \ \ \ \ \ \ \ \ \ \ \ \ \ \ \ \ \ \ \ \
\ \ \ \ \ \ \ \ \ \ \ \ \ \ \ \ \ \ \ \ \ \ \ \ \ \ \ \ \ \ \ \ \ \ \ \ \ \
\ \ \ \ \ \ \ \ \ \ \ \ \ \ \ \ \ \ \ \ \ \ \ \ \ \ \ \ \ \ \ \ \ \ \ \ \ \
\ \ \ \ \ \ \ \ \ \ \ \ \ \ \ \ \ \ \ \ \ \ \ \ \ \ \ \ \ \ \ \ \ \ \ \ \ \
\ \ \ \ \ \ \ \ \ \ \ \ \ \ \ \ \ \ \ \ \ \ \ \ \ \ \ \ \ \ \ \ \ \ \ \ \ \
\ \ \ \ \ \ \ \ \ \ \ \ \ \ \ \ \ \ \ \ \ \ \ \ \ \ \ \ \ \ \ \ \ \ \ \ \ \
\ \ \ \ \ \ \ \ \ \ \ \ \ \ \ \ \ \ \ \ \ \ \ \ \ \ \ \ \ \ \ \ \ \ \ \ \ \
\ \ \ \ \ \ \ \ \ \ \ \ \ \ \ \ \ \ \ \ \ \ \ \ \ \ \ \ 

To start moving toward a diagram of an indecomposable module, we may attach
a vertex to each simple at all layers, and start by joining the vertex
corresponding to the simple head with all the vertices of the second layer,
give the edges an "upward" orientation, and attach to each of them the
corresponding extention class (equivalently, the corresponding elements of $%
Ext_{A}^{1}$).

We may proceed this way, by taking the indecomposable subfactors with the
simple heads of the subsequent layers and join each simple head to the
essential irreducibles of the next radical layer of the relevant
indecomposable sector, but such a process raises many subtle conditions, to
make it possible.

We have a freedom of choice on "essential"\ irreducible summands, whenever
there are multiple isomorphic copies on a layer, resulting in a variation of
the extension classes, according to Theorem . However it is very important
to make a properly coordinated choice of virtual irreducibles on the
different layers, so as to make the process possible, and also get to a
nicely "tuned" diagram, instead of a possibly chaotic picture. But what does
it exactly mean for a diagram to be thus "nice"? We are thus lead to the
notion of a "tuned" diagram:

\bigskip

\begin{definition}
We call a (virtual) diagram of a module $M$, that is associated to a virtual
radical series, a \textbf{tuned} one, if the following condition is
satisfied: For $0\leq i<j\,\,$, any indecomposable summand of the "radical
section" $rad^{i}M/rad^{j}M$ either precisely corresponds to some connected
component of the corresponding radical section of\ the diagram (: "is
visible on the diagram") or is a submodule of a direct sum of visible
isomorphic direct summands of that section, i.e. ones that "virtually
correspond" to (: "are realized by") some connected components of the
corresponding graph section. Such indecomposable summands will be called "$%
\left\{ i,j\right\} $\textbf{-pillars}", of \textbf{height} at most $j-i$.
An $\left\{ i,j\right\} $\textbf{-colonnade} is a maximal direct sum of (at
least two) isomorphic $\left\{ i,j\right\} $-pillars. A pillar with no other
isomorphic pillars in the same radical section is called a \textbf{single
pillar.} A \textbf{specification} \textbf{of a colonnade} is any choice of
specific pillars (of which we are thus also getting an \textbf{"illustration"%
}, as they become \textbf{visible}) for the expression of the colonnade as a
direct sum. A pillar $B$ is "dominated" $\left( \leq \right) $ by another
one $A$,\ if in any diagrammatic illustration of the module $M$, in which
the pillar $A$ becomes visible, the pillar $B$ is depicted by a subdiagram
of the diagram section depicting $A$. In this context we shall be talking of
"maximally/minimally dominating pillars in a (radical) section".
Alternatively (conversely) we shall also describe that condition by telling
that $A$ \textbf{overcoats} $B$. 

A single pillar of a radical section is an indecomposable direct summand
with no other isomorphic summand in the same radical section. A \textbf{peak}
of a colonnade is a single pillar that overcoats a pillar of the colonnade.
As suggested above, we shall call a pillar or, more generally, a subsection
of $M$, a \textbf{visible} one (on the diagram), if it properly (:
virtually) corresponds to a certain subdiagram.

In the case of just a virtual radical or socle series of $M$, we shall say
that a (virtual) section is visible or realizable on it, if all its simple
(virtual) sections are included in it.
\end{definition}

\bigskip So, we might speak of visible sections, hoping also to somehow get
able to gain an overview of the non-visible ones.

Conversely, we might ask: \textit{"When does a "locus" of the virtual
radical/socle series (respectively, of a virtual diagram) correspond to a
section of the module?"} Our term in that case shall be that the subdiagram
(or locus from the virtual radical/socle series) may or may not be \textbf{%
realizable}, where the virtual section that realizes the locus \textit{%
"overcoats"} it. 

\textbf{A fundamental merit of a "good" virtual diagram shall be the optimal
realizability of such loci and, conversely, the visibility of any subsection
of the module or of a "parallel" isomorphic copy of it, meaning that they
belong to the same colonnade.}

\begin{definition}
\bigskip We introduce a "weak" partial ordering "$\sqsubseteq $" in the set $%
\Pi $ of colonnades of a module $M$, meaning that there is a specification
of the "bigger", where at least one pillar dominates a specification of the
"smaller" one. We shall be especially interested in the \textbf{maximal
colonnades} with respect to this ordering.
\end{definition}

\bigskip\ After introducing all these key-notions, we shall prove some
lemmas, that shall lay the way for the proof of the existence theorem for
virtual diagrams.

\bigskip \bigskip 

\begin{lemma}
\bigskip No two maximally dominating pillars of a module $K$ may have a
(virtually!) common section.
\end{lemma}

\begin{proof}
\bigskip The question is easily seen to boil down to the following:

\textit{\textquotedblleft Let }$K$\textit{\ be a module\ of radical length }$%
\kappa +1$\textit{, and }$M$\textit{, }$N$\textit{\ a pair of pillars\ of }$K
$\textit{, where }$M$\textit{\ is a direct summand of }$K/rad^{i+s}K$\textit{%
, }$N$\textit{\ of }$rad^{i}K$\textit{, both of maximal radical length
(resp., }$i+s$\textit{\ and }$\kappa -i+1$\textit{), for some }$i>1$\textit{%
, }$s>0$\textit{\ with }$i+s<\kappa +1$\textit{, and such that }$M$\textit{, 
}$N$\textit{\ have a common section }$S$\textit{, meaning here that }$S$ is
virtually a submodule of $M$\textit{\ and a factor module of }$N$,\textit{\ 
and }$S$\textit{\ in the Loewy series "extends over" all the }$s$\textit{\
layers of }$K$\textit{'s radical section }$rad^{i}K/rad^{i+s}K$\textit{; we
may furthermore easily reduce to the case where }$S=rad^{i}M$\textit{=}$%
N/rad^{s}N\subseteqq rad^{i}K/rad^{i+s}K$\textit{\ (always virtually meant).
Then neither }$M$\textit{\ or }$N$\textit{\ can be a\ maximally dominating
pillar of }$M$\textit{.\textquotedblright }

The condition for $\NEG{M}$ and $N$ does not only suggest the existence of a
homomorphism $\varphi :N\rightarrow M$, where $\ker \varphi =\func{Im}%
\varphi =S$, but here we do also have something analogue to the inclusion,
considered as a special kind of monomorphism. \textit{This analogue becomes
only possible in the frame of our virtual category, in which we have (the
possibility) to see everything also set-theoretically at the same time, and
not just up to isomorphism.}

By the definition of $S$ it is clear that it is realizable in $T$, inasmuch
as $M$ and $N$ are so.

$%
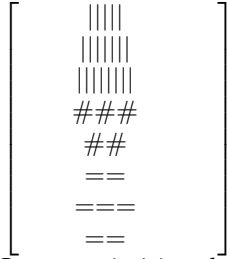
\begin{bmatrix}
&  & ||||| &  &  &  \\ 
&  & ||||||| &  &  &  \\ 
&  & |||||||| &  &  &  \\ 
&  & \#\#\# &  &  &  \\ 
&  & \#\# &  &  &  \\ 
&  & == &  &  &  \\ 
&  & === &  &  &  \\ 
&  & == &  &  & 
\end{bmatrix}%
$

In my primitive drawing above I am trying to visualize the situation by
denoting the location of $S$, shared in common by both $M$ and $N$,\ with
\#\#, that belonging only to M by vertical lines \TEXTsymbol{\vert}%
\TEXTsymbol{\vert}\TEXTsymbol{\vert}\TEXTsymbol{\vert} , the one exclusively
belonging to $N$ by horizontal lines == .

Let us first notice that, by going from the radical to the socle series, the
location of section $S$ shall still extend over "the same" $s$ layers,
because: (a) The socle series length of $S$ and the indecomposables $M$ and $%
N$ remain the same as that of radical series, and (b) nothing can "fall
down", so as to extend the length $i+s$ of $M$ in $K$, because that would
contradict the indecomposability of $M$, unless $M$ as a whole be shifted
downward, taking also $S$ down with it; but that would necessarily decrease
the length $\kappa -i+1$ of the indecomposable $N$, a contradiction.

Set $K/rad^{i+s}K=M\oplus M_{0}$, implying $rad^{i}K/rad^{i+s}K=rad^{i}M%
\oplus rad^{i}M_{0}=S\oplus S_{0}$, where $S_{0}:=$\ $rad^{i}M_{0}$. We
have, further, the following\textit{\ }virtual equalities: $%
K/rad^{i}K=K/rad^{i+s}K\diagup rad^{i}K/rad^{i+s}K=\left( M\oplus
M_{0}\right) \diagup \left( S\oplus S_{0}\right) =M/S\oplus
M_{0}/S_{0}=M^{\prime }\oplus M_{0}^{\prime }$, where we have put $M^{\prime
}:=M/S$ and $M_{0}^{\prime }:=M_{0}/S_{0}$. Similarly, by writing $%
rad^{i}K=N\oplus N_{0}$, $rad^{s}N=N^{\prime }$, $rad^{s}N_{0}=N_{0}^{\prime
}$ we get $rad^{i+s}K=N^{\prime }\oplus N_{0}^{\prime }$.

Consider now the canonical epimorphisms $\pi :K\rightarrow K/rad^{i}K$, $\pi
_{1}:K\rightarrow K/rad^{i+s}K$,$\ \pi _{2}:K/rad^{i+s}K\rightarrow
K/rad^{i}K$; we are now going to utilize the notion of confined preimage.

Clearly $\pi =\pi _{2}\circ \pi _{1}$, therefore $\widetilde{\pi }%
^{-1}\left( M^{\prime }\right) =\widetilde{\pi }_{1}^{-1}\left( \widetilde{%
\pi }_{2}^{-1}\left( M^{\prime }\right) \right) =\widetilde{\pi }%
_{1}^{-1}\left( M\right) $, the intersection of which with $rad^{i+s}K$ has
to be the same as that of $\widetilde{\pi }_{0}^{-1}\left( S\right) $, where 
$\pi _{0}:rad^{i}K\rightarrow rad^{i}K/rad^{i+s}K$, inasmuch as $s>0$, i.e. $%
\widetilde{\pi }_{1}^{-1}\left( M\right) \cap rad^{i+s}K=\widetilde{\pi }%
_{0}^{-1}\left( S\right) \cap rad^{i+s}K=N^{\prime }$. However this means
that $\widetilde{\pi }_{1}^{-1}\left( M\right) $ is a module $\Lambda $,
that precisely overcoats $M$ and $N$ of Loewy length $\kappa +1$ and is
clearly indecomposable; as soon as we prove this also to be a direct summand
of $K$, then it shall be a pillar, that properly dominates over both $M$ and 
$N$,\ contrary to their maximality.

Consider the (virtual)\ exact sequence $0\rightarrow N^{\prime }\oplus
N_{0}^{\prime }\rightarrow K\rightarrow M\oplus M_{0}\rightarrow 0$, where
the first map is an inclusion and the second is $\pi _{1}:K\rightarrow
K/rad^{i+s}K$ $\ \left( =M\oplus M_{0}\right) $; therefore by looking at the
relevant confined preimage we get $\widetilde{\pi }_{1}^{-1}\left( M\right)
\cap rad^{i+s}K=N^{\prime }$, where $rad^{i+s}K=N^{\prime }\oplus
N_{0}^{\prime }$, implies $\widetilde{\pi }_{1}^{-1}\left( M_{0}\right) \cap
rad^{i+s}K=N_{0}^{\prime }$, meaning that $\Lambda $\ is a direct summand of 
$K$, as claimed - leading to the suggested contradiction.

\ \ \ 
\end{proof}

\bigskip It is easily seen that the module $\Lambda $ obtained in the proof
of the previous lemma is indeed the cup product (seen as Yoneda splice, see
f.ex. \cite[6 (8), p. 35]{JC}) of the extensions\ $0\rightarrow N^{\prime
}\rightarrow N\rightarrow M\rightarrow M^{\prime }\rightarrow 0$ \ (where $%
N\rightarrow M$ is the map $\varphi $ above) and \ $0\rightarrow N^{\prime
}\rightarrow N\rightarrow S\rightarrow 0$ \ \ \ $0\rightarrow S\rightarrow
M\rightarrow M^{\prime }\rightarrow 0$ .

The map $\varphi :N\rightarrow M$ above in an intuitively very meaningful
way generalizes two special cases of homomorphisms: namely, inclusion and
canonical epimorphism. Such a generalization is only conceivable in the
frame of a virtual category; we shall call them "canonical homomorphisms": 

\begin{definition}
Given two virtual sections  $M$ and $N$ of a module $K$ and a third one, say 
$S$, that is embedded in $M$ and on which we, on the other hand, have a
canonical epimorphism $N\twoheadrightarrow S$, we call their composition $%
N\twoheadrightarrow S\hookrightarrow M$ \textbf{a canonical homomorphism
(over }$S$\textbf{). }We may as well define the \textbf{"strictly virtual
category"}, in which only canonical homomorphisms are allowed as morphisms.
\end{definition}

\bigskip It is obvious from the definition that inclusions and canonical
epimorphisms are just special cases of canonical homomorphisms.

\begin{lemma}
If $L_{1}$ and $L_{2}$ are visible summands of some radical (/socle) series
sections, of Loewy length $l_{1}$ and $l_{2}$ respectively,\ with only one
(visible) common section at their top, i.e. a common factor $T$ (not
necessarily simple!), then the pullback of the canonical epimorphisms $%
L_{1}\rightarrow T$, $L_{2}\rightarrow T$ is visible too.

More generally, if we have canonical epimorphisms  $L_{i}\rightarrow T$,\ $%
i=1,...,s$, with all $L_{i}$ and $T$ visible, then they determine uniquely a
section, that precisely overcoats them.
\end{lemma}

\begin{proof}
\bigskip Assume first $l_{1}$=$l_{2}$; then their sum $L_{1}+L_{2}$, making
sense inside the least radical series section that overcoats them, is
actually the pullback. If $l_{1}$\TEXTsymbol{<}$l_{2}$,\ amputate $L_{2}$ to
get to the first case and then use the tachnic of the previous proof to get
the section that precisely overcoats $L_{1}$ and $L_{2}$.
\end{proof}

\bigskip Observe that in case $l_{1}\neq l_{2}$,\ we cannot make any sense
of the sum $L_{1}+L_{2}$ in the given virtual category, therefore we
amputate the longest (say, $L_{2}$) down to, say, $L_{0}$, make use of the
first case of the proof of this Lemma to get a "virtual pullback" $L$ and,
subsequently, get the ("canonical", while all maps involved are so) cup
product over $L_{0}$ of $L$ and $L_{2}$. That is actually the pullback of
the canonical epimorphisms $L_{1}\rightarrow T$, $L_{2}\rightarrow T$. 

By dualizing the arguments, we get the dual of the above:

\begin{lemma}
If $L_{1}$ and $L_{2}$ are visible summands of some radical (/socle) series
sections, of Loewy length $l_{1}$ and $l_{2}$ respectively,\ with only one
(visible) common section at their bottom, i.e. a common virtual submodule $T$%
, then the pushout of the embeddings $T\rightarrow L_{1}$, $T\rightarrow
L_{2}$ is visible too.

More generally, if we have embeddings $T\rightarrow L_{i}$,\ $i=1,...,s$,
with all $L_{i}^{{}}$'s and $T$ visible, then they determine uniquely a
section, that precisely overcoats them.
\end{lemma}

\bigskip 

\begin{proposition}
We may to every finitely generated $A$-module $M$ attach a tuned radical
diagram.
\end{proposition}

\begin{proof}
\ \ \textit{SKETCH OF PROOF:}

\bigskip Clearly, it suffices to show it for an indecomposable module $M$.

We proceed to construct such a diagram:

Take all peaks and all maximal colonnades. Prove that they cannot overlap.
Investigate possible differentiation of some realizations of maximal
colonnades, prefer the ones that highlight them. "VERTICAL PRIORITY".

Consider then for any "$\sqsubseteq $"-maximal $\left\{ i,j\right\} $%
-colonnade the "$\leq $"-maximal (:maximally dominating) indecomposable $%
\left\{ i-1,j+1\right\} $-overcoat (where, if either $i-1$ or $j+1$ layer
non-existing, take $i$, resp. $j$ instead) precisely overcoating the
colonnade on the respective radical (socle) series section, call them "%
\textbf{pillar bricks}"; it follows easily that they can only overlap on
their endlayers, same or opposite direction (use maximality of colonnades).

Our construction ensures that we do not mess with pillar-realizations,
ensuring "vertical priority", in the sense that the "broader" (radical)
section consideration comes before the lesser, narrower ones in the
realization process.

Inside each brick, we consider the maximal indecomposables with one of the
(virtually determined) irreducibles of the head as their simple head. We
repeat the process in narrower radical sections inside the radical section
of each maximal pillar, as many times as possible: That means, repeating the
process inside the pillars of the maximal colonnades.

By taking all the indecomposables with simple head, we join the simple head
by an edge with all irreducibles on its next radical layer, while attaching
the proper extension-class to it.

Thus we get a diagram, which can easily be shown to be an "optimally tuned"
one, in the sense of "vertical priority".
\end{proof}

\bigskip \bigskip 

\begin{conclusion}
We achieve a final \textbf{diagram} out of a "radical" one in the following
way:
\end{conclusion}

Inside the virtual category of a given module $M$, we may then identify the
virtual irreducibles of the radical series as contents of the socle series
too, not necessarily at the same layers: The virtual irreducibles that are
"hanging loosely" from the preceding radical layer shall fall down to a
lower layer of the socle series. The key background property for this
situation is the existence in a radical section of height $s$ of a direct
summand of Loewy length smaller than $s$. These provide the missing edges of
the module diagram, by considering (dually now) the indecomposable sections
with a simple socle in the subsequent socle layer; those indecomposables are
the same as the ones we considered through the virtual radical series,
inside the categories of modules we are considering. Thus we get to the
important

\begin{theorem}
\bigskip We may to every finitely generated $A$-module $M$ attach a certain
type of a tuned diagram, to be called \textbf{a central diagram}. In
particular, there is always a tuned diagram.
\end{theorem}

Thinking of the virtual radical, resp. socle, series in terms of compatible
filtrations, we may read such filtrations either "downward", i.e. from the
bigger to the smaller (as in radical series), or upward (like the socle
series). In the former case, a step is non-split whenever the next term
looses a (virtual) irreducible of the next radical layer; dually in the
latter case, a step upward is non-split whenever the next term gains a
(virtual) irreducible of the next (:higher) socle layer.

\bigskip

\bigskip It is not difficult to prove that it is possible to make a diagram
that somehow highlights (: makes it visible) any particular section of the
module. A lot of relevant questions arise: Can we consider the totality of
tuned diagrams for a module? May some of them "lead" to all others - or to a
family of them?

\bigskip \bigskip The Theorem of existence of a tuned diagram can become a
powerful tool in many situations. In my following research I shall try to
illustrate this by some important examples.

It may for example become a powerful tool in many cases where we restrict
modules to subcategories/subgroups, also depending on the other information
we may have, in particular about the restriction of the Ext-functors and
other relevant specific knowledge.

We close this section by stressing that the present article has not yet
reached its final edition.

\section{\protect\bigskip An important field of application.}

\bigskip

\bigskip Here we are now going to outline our strategy for the determination
of virtual diagrams for the p.i.m.'s of the family of groups $SL(3,p)$, with 
$p>3$. That strategy may be pursued in the more general case of the above
mentioned  truncated category $C\left( \pi \right) $, closely related to the
representation theory of finite groups of Lie type. 

Particularly important in our considerations is the following result on
truncated categories (\cite[A.10, p. 393]{JCJ}): 

\begin{proposition}
Let $\pi \subset X\left( T\right) _{+}$ be saturated, $V$, $V^{\prime }$ be $%
G-$modules in $C\left( \pi \right) $. Then for all $i\geq 0$ there are
isomorphisms $Ext_{C\left( \pi \right) }^{i}\left( V,V^{\prime }\right)
\cong Ext_{G}^{i}\left( V,V^{\prime }\right) $.
\end{proposition}

Even more important for our goals is the following Proposition, which was
already proved in \cite[Theorem 7.4]{CPSK}, but of which Henning Haahr
Andersen in \cite[Proposition 2.7]{HHA} gave a new elegant and
self-contained proof:

\begin{proposition}
Let $G$ be an almost simple algebraic group over an algebraically closed
field $k$\ of positive characteristic $p$, and assume that $G$\ is defined
and split over the prime field $\digamma _{p}$\bigskip ; we denote as $%
G\left( n\right) $ the finite group consisting of the points of $G$ over the
field with $p^{n}$ elements. \bigskip Then \ for $\lambda $,$\mu \in
X_{n}\left( T\right) $, the restriction map $Ext_{G}^{1}\left( L\left( \mu
\right) ,L\left( \lambda \right) \right) \longrightarrow Ext_{G\left(
n\right) }^{1}\left( L\left( \mu \right) ,L\left( \lambda \right) \right) $
is injective.
\end{proposition}

\bigskip However, by going through that proof we can confirm that, with a
slight adaptation, we deduce the following stronger version, about which it
has though to be remarked that it only makes full sense in the frame of the
virtual category of the $G-$module $E$ (and its restriction down to $G\left(
n\right) $):

\begin{proposition}
Same notation as above, except that we now allow for any $\lambda $,$\mu \in
X\left( T\right) $, instead. Let $0\rightarrow L\left( \lambda \right)
\rightarrow E\rightarrow L\left( \mu \right) \rightarrow 0$ be a non-split $%
G-$extension. Then is $Hom_{G(n)}(L(%
\mu
),E)=0$, which implies that no simple composition factors of $E$ as a $%
G\left( n\right) -$module, that stem from the restriction of $L\left( \mu
\right) $, can be in the socle. This has to be understood in the context of
the virtual socle series of the $G-$module $E$ and its restriction down to $%
G\left( n\right) $.  
\end{proposition}

\bigskip 

This last propostition is going to be very important in our pursuit of
diagrams for modular representations of finite groups of Lie type, as it is
giving us a kind of thread coming by restriction from the diagrams for
representations (typically the Weyl modules or their duals) algebraic groups
down to the groups $G\left( n\right) $. We are then in many cases able to
get some "good filtrations" of tensor products of suitable simple modules
with the Steinberg module $St_{n}=H^{0}\left( \left( p^{n}-1\right) \rho
\right) $, of which the corresponding $G_{n}-$injective hull $Q_{n}\left(
\lambda \right) $ is proven to be a direct summand, by means of the well
known $G_{n}-$ (and $G\left( n\right) -$)injectivity of $St_{n}$, and also
of the fact that $Q_{n}\left( \lambda \right) $ is known to have a $G-$%
module structure for all $\lambda \in X_{n}\left( T\right) $ for all $p\geq
2\left( h-1\right) $, where $h$\ is the Coxeter\ number. 

More precisely, we see that $Hom_{G}(L(\lambda ),St_{n}\otimes L(\left(
p^{n}-1\right) \rho -\lambda ^{\ast }))\cong Hom_{G}(L(\lambda )\otimes
L(\left( p^{n}-1\right) \rho -\lambda ),St_{n})\cong k$, which by virtue of
the $G-$structure of $Q_{n}\left( \lambda \right) $\ and the $G_{n}-$%
injectivity of $St_{n}$, yields that $Q_{n}\left( \lambda \right) $ is a
(single) direct summand of $St_{n}\otimes L(\left( p^{n}-1\right) \rho
-\lambda ^{\ast })$. Next we find a good filtration of this last tensor
product and try to decompose it as a direct sum of indecomposables, one of
which of course is $Q_{n}\left( \lambda \right) $. 

Then again $Q_{n}\left( \lambda \right) $, as it is well known, is also $%
G\left( n\right) -$injective, therefore the $G\left( n\right) -$injective
hull $U_{n}\left( \lambda \right) $\ of $L(\lambda )$\ is a $G\left(
n\right) -$direct summand of $Q_{n}\left( \lambda \right) $. 

In the procedure of restriction we make extensive use of Steinberg's tensor
product theorem; then we shall have to calculate lots of tensor products of
simple $G-$modules, which may either be done ad hoc, for example by weight
considerations or by use of known results on the translation functors or,
fortunately, in the case of $SL_{3}$, by using the algorithms found in \cite%
{BDM}.\ 

Notably for $SL_{3}$\ again, in the mentioned paper \cite{HHA} both the
dimensions of the modules $Ext_{G\left( 1\right) }^{1}$\ and the $G\left(
1\right) -$decomposition of the $Q_{1}\left( \lambda \right) $'s\ have been
explicitly calculated. Furthermore and very importantly, in \cite{DS} and
much more thoroughly in \cite{KKH} we find an explicit calculation of the
submodule-structure of the Weyl modules (and their duals); notice that the
diagrams given there are in reality virtual, since there are nowhere in them
isomorphic irreducibles on the same layer. This convenience does anyway no
longer hold after the mentioned restrictions.

\textit{The last Proposition above shall then be of crucial importance to
articulate the way the virtual contents of the old break up and organize
into the new and far more complicated diagrams, after the restrictions.}

\bigskip\ 

\bigskip I have been working with such examples already as a young student
in the late '80's; my work was then interrupted, due to outer conditions and
personal obligations. After resuming my mathematical activity recently, I
have found in \cite{BC} the only diagrammatic method that I know of, that
somehow reminds of my own approach. There is, however, a fundamental
difference from mine in theirs: Contrary to what I am doing, they start with
defining abstract diagrams, and then they are looking for modules
"representing" them, if any. Nonetheless that article is important, also
from my point of view, especially whenever such a diagram as theirs has a
unique representation (\cite[5.1]{BC}).

\end{document}